\documentclass[12pt]{article}
\usepackage{amsmath,amsfonts,amssymb,amsthm}
\newcommand{\Z}{\mathbb{Z}}

\newcommand{\ca}{\mathcal A}
\newcommand{\twosum}[2]{\sum_{\substack{#1\\#2}}}

\newtheorem{thm}{Theorem}[section]
\newtheorem{lem}[thm]{Lemma}
\begin{document}
\title{Almost-prime values of polynomials at prime arguments}
\author{A.J. Irving\\
Centre de recherches math\'ematiques, Universit\'e de Montr\'eal}
\date{}

\maketitle

\begin{abstract}
We consider almost-primes of the form $f(p)$ where $f$ is an irreducible polynomial over $\mathbb Z$ and $p$ runs over primes.  We improve a result of Richert for polynomials of degree at least $3$.  In particular we show that, when the degree is large, there are infinitely many primes $p$ for which $f(p)$ has at most $\deg f+O(\log\deg f)$ prime factors. 
\end{abstract}

\section{Introduction}

A well known problem in number theory is to show that if the values $f(n)$ taken by an irreducible polynomial $f\in\Z[x]$ have no fixed prime divisor then there are infinitely many $n$ for which $f(n)$ is prime.  This is known for polynomials of degree $1$, in which case it is Dirichlet's theorem on primes in arithmetic progressions, but it is an open problem for $f$ of higher degree.  By using sieve methods one can prove a weaker statement, namely that there are infinitely many $n$ for which the values $f(n)$ have a bounded  number of prime factors.  If we let $P_r$ denote numbers with at most $r$ prime factors, counted with multiplicities, then Richert \cite[Theorem 6]{richert} showed that $f(n)$ is infinitely often a $P_{\deg f+1}$.  This was improved for a quadratic polynomial by Iwaniec \cite{iwaniecn21}, who showed that $n^2+1$ is a $P_2$ for infinitely many $n$.  

In this paper we will consider  $f(p)$ as $p$ runs over primes.  Assuming they have no fixed prime divisor it is conjectured that these values contain infinitely many primes.  This is not even known for linear polynomials, a special case of which would be the famous twin prime conjecture.  Concerning almost-primes, it was shown by Richert \cite[Theorem 7]{richert} that $f(p)$ is infinitely often a $P_{2\deg f+1}$.  The only known improvement to this result is due to Chen \cite{chen} who showed that we can find infinitely many $P_2$ when $\deg f=1$.
  The aim of this work is to improve Richert's result on the values $f(p)$ for all $f$ with $\deg f\geq 3$.

\begin{thm}\label{mainthm}
Suppose $f\in\Z[x]$ is an irreducible polynomial with a positive leading coefficient and that for all primes $p$ we have 
\begin{equation}\label{localcond}
\#\{a\pmod p:(a,p)=1\text{ and }f(a)\equiv 0\pmod p\}<p-1.
\end{equation}
Then, for all sufficiently large $x$, we have 
$$\#\{p\leq x:f(p)\in P_r\}\gg_{f,r} \frac{x}{(\log x)^2},$$
provided that $r\geq r(\deg f)$.  The values of $r(\deg f)$ for small degree are given by Table \ref{smalldeg}.  When $\deg f$ is sufficiently large they satisfy 
$$r(\deg f)=\deg f+c\log\deg f+O(1),$$
with $c=3.120\ldots.$
\end{thm}

\begin{table}
\centering
\label{smalldeg}
  \begin{tabular}{r|rrrrrrrr}
    $\deg f$ & 3 & 4 & 5 & 6 & 7 & 8 & 9 & 10\\
\hline
$r$&  6 & 8 & 10 & 11 & 12 & 14 & 15 & 16\\ 
  \end{tabular}
\caption{Values of $r(\deg f)$ for small degrees}
\end{table}

The most significant feature of Theorem \ref{mainthm} is that the factor $2$ in Richert's result has been removed, at the cost of an additional lower-order term.  In particular, for polynomials of large degree, we have come within $O(\log\deg f)$ of getting the same result for $f(p)$ as Richert's for $f(n)$.  The results for small degree are less striking, and we cannot improve on $r=5$ for $\deg f=2$.

Richert's result, $r=2\deg f+1$, 
makes crucial use of the Bombieri-Vinogradov theorem. He is therefore limited to sieving out primes up to at most $\sqrt{x}$ and thus he cannot achieve $r<2\deg f$.  Under the assumption of the Elliott-Halberstam conjecture \cite{elliott}, which gives a best possible level of distribution for the primes in arithmetic progressions, Richert's argument could be modified to yield $r=\deg f+1$.  In this work we use the same weighted sieve as Richert.  However, we will show in Lemma \ref{lem2dim} that some terms in the resulting sum may be estimated more efficiently using a $2$-dimensional sieve, rather than with a $1$-dimensional sieve and the Bombieri-Vinogradov theorem.  The result is that we may sieve by primes which are almost as large as $x$.

\subsection*{Acknowledgements}

This work was completed whilst I was a CRM-ISM postdoctoral fellow at the Universit\'e de Montr\'eal.  I am grateful to Andrew Granville and Roger Heath-Brown for some useful suggestions.

\section{Notation and Preliminaries}

Throughout this work $p$ will always denote a prime variable.  For a large $x$ we will apply a sieve to the set 
$$\ca=\{f(p):p\in (x,2x]\}.$$
All uses of the symbols $o$ and $\sim$ will be as $x\rightarrow\infty$ and any inequality using $\gg$ will be assumed only to hold for all sufficiently large $x$.  All our implied constants may depend on the polynomial $f$.  To prove Theorem \ref{mainthm} it is sufficient to show that 
$$\#(\ca\cap P_r)\gg \frac{x}{(\log x)^2}.$$
We will assume that $f$ has a nonzero constant term, since the only irreducible $f$ for which this does not hold are $f(x)=ax$ and the result is trivial in that case.

For convenience we will write $k=\deg f$.  We let 
$N=\max\ca$
and note that this satisfies 
$$N\sim f_kx^k,$$
where $f_k$ is the leading coefficient of $f$ which we are assuming is positive.  We write 
$$X=\#\ca=\pi(2x)-\pi(x),$$
where $\pi(x)$ is the usual prime counting function.
  It follows by the Prime Number Theorem that 
$$X\sim \frac{x}{\log x}.$$

We will use the arithmetic functions 
$$\nu_1(d)=\#\{a\pmod d:(a,d)=1\text{ and }f(a)\equiv 0\pmod d\}$$
and 
$$\nu_2(d)=\#\{a\pmod d:af(a)\equiv 0\pmod d\}.$$
By the Chinese Remainder Theorem both $\nu_1$ and $\nu_2$ are multiplicative.  We note that if $p\nmid f(0)$ then 
$$\nu_1(p)=\#\{a\pmod p:f(a)\equiv 0\pmod p\}.$$
 In addition, since $\nu_1(p)$ never counts $0\pmod p$, we see that for all $p$ we have 
\begin{equation}\label{nu1nu2}
\nu_2(p)=\nu_1(p)+1.
\end{equation}
The next lemma gives asymptotic formulae for the various sums and products of $\nu_1$ and $\nu_2$ which we will need.

\begin{lem}\label{sumprods}
For $x\geq 2$ we have 
\begin{equation}\label{nu1sum}
\sum_{p\leq x}\frac{\nu_1(p)\log p}{p-1}=\log x+O(1),
\end{equation}
\begin{equation}\label{nu2sum}
\sum_{p\leq x}\frac{\nu_2(p)\log p}{p}=2\log x+O(1),
\end{equation}
\begin{equation}\label{nu1prod}
\prod_{p\leq x}\left(1-\frac{\nu_1(p)}{p-1}\right)\sim \frac{c_f}{\log x}
\end{equation}
and 
\begin{equation}\label{nu2prod}
\prod_{p\leq x}\left(1-\frac{\nu_2(p)}{p}\right)\sim \frac{e^{-\gamma}c_f}{(\log x)^2},
\end{equation}
where $c_f$ denotes a constant which is positive assuming (\ref{localcond}).
\end{lem} 

\begin{proof}
Since $f$ is irreducible it can be shown using ideas from algebraic number theory, see for example Diamond and Halberstam \cite[Proposition 10.1]{diamondhalberstam}, that 
\begin{equation}\label{nu1standard}
\sum_{p\leq x}\frac{\nu_1(p)\log p}{p}=\log x+O(1).
\end{equation}
However 
$$\sum_{p\leq x}\nu_1(p)\log p\left(\frac{1}{p-1}-\frac{1}{p}\right)= \sum_{p\leq x}\frac{\nu_1(p)\log p}{p^2}\ll \sum_{p\leq x}\frac{\log p}{p^2}\ll 1$$
so (\ref{nu1sum}) follows.  Combining (\ref{nu1standard}) with (\ref{nu1nu2}) and the Mertens estimate 
$$\sum_{p\leq x}\frac{\log p}{p}=\log x+O(1)$$
one can derive (\ref{nu2sum}).

The deduction of (\ref{nu1prod}) from (\ref{nu1sum}) is standard.  To prove (\ref{nu2prod}) we write 
$$1-\frac{\nu_2(p)}{p}=1-\frac{\nu_1(p)+1}{p}=\left(1-\frac{1}{p}\right)\left(1-\frac{\nu_1(p)}{p-1}\right).$$
We may therefore deduce (\ref{nu2prod}) from (\ref{nu1prod}) and the Mertens estimate 
$$\prod_{p\leq x}\left(1-\frac{1}{p}\right)\sim \frac{e^{-\gamma}}{\log x}.$$
\end{proof}

We will use the standard sieve theory notation 
$$\ca_d=\{n\in\ca:d|n\}$$
and 
$$S(\ca,z)=\#\{n\in \ca:(n,P(z))=1\},$$
where 
$$P(z)=\prod_{p<z}p.$$
We denote Euler's totient function by $\varphi(d)$.  The next lemma gives a level of distribution for $\ca$, showing that on average over sufficiently small $d$ we have 
$$\#\ca_d\approx \frac{X\nu_1(d)}{\varphi(d)}.$$
We therefore let 
$$R_d=\#\ca_d-\frac{X\nu_1(d)}{\varphi(d)}.$$

\begin{lem}\label{lod}
Suppose $0<\theta_1<\frac{1}{2k}$ and $A>0$.  We have 
$$\twosum{d\leq N^{\theta_1}}{\mu(d)\ne 0}|R_d|\ll_{A,\theta_1} x(\log x)^{-A}.$$
\end{lem}

\begin{proof}
This may be deduced from the Bombieri-Vinogradov theorem, the details are given by Richert as part of the proof of \cite[Theorem 7]{richert}.
\end{proof}

\section{The Weighted Sieve}

Apart from some changes of notation, we will use the same weighted sieve as Richert \cite{richert}.  For fixed $k$ and $r$ let $0<\alpha<\beta$ be constants to be chosen later and let 
$$z=N^\alpha,\qquad y=N^\beta.$$
We define 
$$\eta=r+1-\frac{1}{\beta}$$
so that $\eta>0$ whenever 
$$\beta>\frac{1}{r+1}.$$
We consider the sum 
$$S=S(\ca,r,\alpha,\beta)=\twosum{n\in\ca}{(n,P(z))=1}w(n)$$
where 
$$w(n)=1-\frac{1}{\eta}\twosum{p|n}{z\leq p<y}w_p$$
with 
$$w_p=1-\frac{\log p}{\log y}.$$

\begin{lem}\label{lemweighted}
Suppose that for a given $k,r$ there exist constants $\alpha,\beta$ with 
$$0<\alpha<\beta<\frac{1}{k}$$
and 
$$\beta>\frac{1}{r+1}$$
for which the estimate 
\begin{equation}\label{Sbound}
S\gg_{\alpha,\beta,r} \frac{x}{(\log x)^2}
\end{equation}
holds.  We may then conclude that 
$$\#(\ca\cap P_r)\gg_r \frac{x}{(\log x)^2}.$$
\end{lem}

\begin{proof}
Since $\beta>\frac{1}{r+1}$ we have $\eta>0$ and therefore $w(n)\leq 1$ for all $n\in\ca$.  The bound (\ref{Sbound}) therefore implies that 
$$\#\{n\in\ca:(n,P(z))=1,w(n)>0\}\gg \frac{x}{(\log x)^2}.$$
However, if $(n,P(z))=1$ and $w(n)>0$, we have 
$$1-\frac{1}{\eta}\twosum{p|n}{p<y}\left(1-\frac{\log p}{\log y}\right)>0$$
which implies that 
$$\twosum{p|n}{p<y}\left(1-\frac{\log p}{\log y}\right)<r+1-\frac{1}{\beta}.$$
If $p\geq y$ then $1-\frac{\log p}{\log y}\leq 0$ so we deduce that 
$$\twosum{p|n}{p<y}\left(1-\frac{\log p}{\log y}\right)+\twosum{p^a|n}{p\geq y}\left(1-\frac{\log p}{\log y}\right)<r+1-\frac{1}{\beta}.$$
We conclude that 
$$\#\{p|n:p<y\}+\#\{p^a|n:p\geq y\}<r+1-\frac{1}{\beta}+\frac{\log n}{\log y}\leq r+1.$$
We have therefore shown that $\ca$ contains $\gg\frac{x}{(\log x)^2}$ numbers, all of whose prime factors are at least $z$, for which 
$$\#\{p|n:p<y\}+\#\{p^a|n:p\geq y\}\leq r.$$

To complete the proof we must show that prime factors in $[z,y)$ can be counted with multiplicity.  We therefore estimate 
\begin{eqnarray*}
\sum_{z\leq p<y}\#\ca_{p^2}&\leq&\sum_{z\leq p<y}\#\{n\in (x,2x]:p^2|f(n)\}\\
&\ll&\sum_{z\leq p<y}\#\{a\pmod{p^2}:f(a)\equiv 0\pmod{p^2}\}(\frac{x}{p^2}+1)\\
&\ll_f&\sum_{z\leq p<y}(\frac{x}{p^2}+1)\\
&\ll&\frac{x}{z}+y=o(\frac{x}{(\log x)^2}),\\
\end{eqnarray*}
where the last inequality follows since $\alpha>0$ and $\beta<\frac{1}{k}$.  We conclude that the contribution to our count from those $n$ divisible by the square of a prime from $[z,y)$ is sufficiently small and so, since $n$ is not divisible by any prime $p<z$, the result follows.
\end{proof}

In order to prove Theorem \ref{mainthm} it remains to show that, for the given $r$, we can choose suitable $\alpha$ and $\beta$ for which (\ref{Sbound}) can be established.  We begin by writing 
$$S=S(\ca,z)-\frac{1}{\eta}\sum_{z\leq p<y}w_pS(\ca_p,z).$$
We let $u=N^\delta$ for some $\delta\in (\alpha,\beta)$ and split the sum at $u$ to get 
$$S=S(\ca,z)-\frac{1}{\eta}\left(\sum_{z\leq p<u}w_pS(\ca_p,z)+\sum_{u\leq p<y}w_pS(\ca_p,z)\right).$$
In the next section we will give a lower bound for the first term of this and upper bounds for the two sums.  

\section{Sieve Estimates}

We will use both the $1$ and $2$-dimensional forms of the beta-sieve, as described by Friedlander and Iwaniec in \cite[Theorem 11.13]{opera}.  We therefore let $F_\kappa$ and $f_\kappa$ denote the upper and lower bound sieve functions in dimension $\kappa$.  We begin by estimating $S(\ca,z)$ by means of a $1$-dimensional sieve of level $N^{\theta_1}$ for $\theta_1<\frac{1}{2k}$. The estimate (\ref{nu1sum}) and Lemma \ref{lod} show that this is permissible and we get 
$$S(\ca,z)\geq XV(z)\left(f_1\left(\frac{\theta_1}{\alpha}\right)+o(1)\right),$$
where 
$$V(z)=\prod_{p<z}\left(1-\frac{\nu_1(p)}{\varphi(p)}\right).$$   
Next we estimate the sum over $z\leq p<u$.

\begin{lem}\label{lem1dim}
If $0<\alpha<\delta<\theta_1$ then 
$$\sum_{z\leq p<u}w_p S(\ca_p,z)\leq XV(z)\left(\int_\alpha^\delta\left(\frac{1}{s}-\frac{1}{\beta}\right)F_1\left(\frac{\theta_1-s}{\alpha}\right)\,ds+o(1)\right).$$
\end{lem}

\begin{proof}
Since $\delta<\theta_1$ we have $\frac{N^{\theta_1}}{p}\geq 1$ for each $p<u$.  We may therefore apply a $1$-dimensional upper bound sieve of level $\frac{N^{\theta_1}}{p}$ to each $S(\ca_p,z)$.  Observe that if $p\geq z$ and $d|P(z)$ then $(d,p)=1$.  We therefore have 
$$S(\ca_p,z)\leq \frac{X\nu_1(p)}{\varphi(p)}V(z)\left(F_1(s_p)+O\left(\left(\log\frac{N^{\theta_1}}{p}\right)^{-\frac{1}{6}}\right)\right)+O\left(\twosum{d\leq N^{\theta_1/p}}{\mu(d)\ne 0}|R_{pd}|\right),$$  
where 
$$s_p=\frac{\log(N^{\theta_1}/p)}{\log z}.$$
Since $p< N^\delta$ and $\delta<\theta_1$ we have 
$$\left(\log\frac{N^{\theta_1}}{p}\right)^{-\frac{1}{6}}\leq ((\theta_1-\delta)\log N)^{-\frac{1}{6}}=O_{\theta_1,\delta}((\log x)^{-\frac{1}{6}}).$$
In addition, applying Lemma \ref{lod} gives 
$$\sum_{z\leq p<u}w_p\twosum{d\leq N^{\theta_1/p}}{\mu(d)\ne 0}|R_{pd}|\leq \twosum{d\leq N^{\theta_1}}{\mu(d)\ne 0}|R_d|\ll_{A,\theta_1}x(\log x)^{-A}$$
so that 
$$\sum_{z\leq p<u}w_p S(\ca_p,z)\leq XV(z)\sum_{z\leq p<u}w_p\frac{\nu_1(p)}{\varphi(p)}(F_1(s_p)+O((\log x)^{-\frac{1}{6}}))+O_{A,\theta_1}(x(\log x)^{-A}).$$
By taking $A$ large enough and using the estimate (\ref{nu1prod}) we see that 
$$x(\log x)^{-A}=o(XV(z)).$$
In addition we have 
$$\sum_{z\leq p<u}w_p\frac{\nu_1(p)}{\varphi(p)}(\log x)^{-\frac{1}{6}}\ll (\log x)^{-\frac{1}{6}}\sum_{z\leq p<u}\frac{1}{p}=o(1)$$
so it only remains to evaluate 
$$\sum_{z\leq p<u}w_p\frac{\nu_1(p)}{\varphi(p)}F_1(s_p)=\sum_{z\leq p<u}\frac{\nu_1(p)\log p}{p-1}g(p),$$
where
$$g(t)=\left(\frac{1}{\log t}-\frac{1}{\log y}\right)F_1\left(\frac{\log(N^{\theta_1}/t)}{\log z}\right).$$
Summing by parts and applying the estimate (\ref{nu1sum}) then gives 
$$\sum_{z\leq p<u}w_p\frac{\nu_1(p)}{\varphi(p)}F_1(s_p)=(\log u-\log z+O(1))g(u)-\int_z^u (\log t-\log z+O(1))g'(t)\,dt.$$
To estimate the contribution to this from the error terms we observe that 
$$g(u)\ll\frac{1}{\log N}=o(1)$$
and 
$$\int_z^u |g'(t)|\,dt=\int_\alpha^\delta\left|\frac{dg(N^s)}{ds}\right|\,ds.$$
However 
$$g(N^s)=\left(\frac{1}{s\log N}-\frac{1}{\beta\log N}\right)F_1\left(\frac{\theta_1-s}{\alpha}\right)$$
so 
$$\frac{dg(N^s)}{ds}\ll \frac{1}{\log N}$$
and thus 
$$\int_z^u|g'(t)|\,dt=o(1).$$
It follows that 
\begin{eqnarray*}
\sum_{z\leq p<u}w_p\frac{\nu_1(p)}{\varphi(p)}F_1(s_p)&=&(\log u-\log z)g(u)-\int_z^u (\log t-\log z)g'(t)\,dt+o(1)\\
&=&\int_z^u\frac{g(t)}{t}\,dt+o(1)\\ 
&=&\int_z^u\frac{1}{t}\left(\frac{1}{\log t}-\frac{1}{\log y}\right)F_1\left(\frac{\log(N^{\theta_1}/t)}{\log z}\right)\,dt+o(1)\\
&=&\int_\alpha^\delta\left(\frac{1}{s}-\frac{1}{\beta}\right)F_1\left(\frac{\theta_1-s}{\alpha}\right)\,ds+o(1).\\
\end{eqnarray*}
\end{proof}

Next we use the $2$-dimensional sieve to estimate $S(\ca_p,z)$.  It is this result which enables us to take $r<2k$.  

\begin{lem}\label{lem2dim}
Suppose $\theta_2,\alpha<\frac{1}{k}$ and $\beta<\theta_2$.  For any $p\in [z,y)$ we have 
$$S(\ca_p,z)\leq \frac{x\nu_1(p)}{p}V_2(z)(F_2(s'_p)+O((\log x)^{-\frac{1}{6}}))$$
where 
$$V_2(z)=\prod_{p<z}\left(1-\frac{\nu_2(p)}{p}\right)$$
and 
$$s'_p=\frac{\log(N^{\theta_2/p})}{\log z}.$$
\end{lem}

\begin{proof}
For $p\geq z$ we have 
$$S(\ca_p,z)\leq \#\{n\in (x,2x]:p|f(n),(nf(n),P(z))=1\}.$$
We will therefore apply an upper bound sieve to the set 
$$\{nf(n):n\in (x,2x]:p|f(n)\}.$$
If $d|P(z)$ then 
$$\#\{n\in (x,2x]:p|f(n),d|nf(n)\}=\twosum{a\pmod{dp}}{f(a)\equiv 0\pmod p,af(a)\equiv 0\pmod d}\left(\frac{x}{pd}+O(1)\right).$$
Since $(d,p)=1$ it follows by the Chinese Remainder Theorem that 
\begin{eqnarray*}
\lefteqn{\#\{a\pmod{pd}:f(a)\equiv 0\pmod p,af(a)\equiv 0\pmod d\}}\\
\hspace{3cm}&=&\#\{a\pmod p:f(a)\equiv 0\pmod p\}\nu_2(d)=\nu_1(p)\nu_2(d),\\
\end{eqnarray*}
where we have used the fact that $p$ is large so $f(0)\not\equiv 0\pmod p$.  For any $\epsilon>0$ we conclude, using the bounds 
$$\nu_1(p)\ll 1 \text{ and }\nu_2(d)\ll_\epsilon d^\epsilon,$$
 that 
$$\#\{n\in (x,2x]:p|f(n),d|nf(n)\}=\frac{x\nu_1(p)\nu_2(d)}{pd}+O_\epsilon(d^\epsilon).$$
The estimate (\ref{nu2sum}) shows that the $2$-dimensional sieve may be applied with density function $\frac{\nu_2(d)}{d}$.  We do so, with level $N^{\theta_2}/p$, to deduce that 
$$S(\ca_p,z)\leq \frac{x\nu_1(p)}{p}V_2(z)\left(F_2(s'_p)+O\left(\left(\log\frac{N^{\theta_2}}{p}\right)^{-\frac{1}{6}}\right)\right)+O_\epsilon\left(\frac{N^{\theta_2+\epsilon}}{p}\right).$$
Since $\beta<\theta_2$ it follows that for $p\leq y$ we have 
$$\left(\log\frac{N^{\theta_2}}{p}\right)^{-\frac{1}{6}}\leq ((\theta_2-\beta)\log N)^{-\frac{1}{6}}=O_{\theta_2,\beta}((\log x)^{-\frac{1}{6}}).$$
In addition, since $\theta_2<\frac{1}{k}$ we can take a sufficiently small $\epsilon$ and use (\ref{nu2prod}) to obtain 
$$N^{\theta_2+\epsilon}\ll xV_2(z)(\log x)^{-\frac{1}{6}}.$$
\end{proof}

We sum this estimate over $u\leq p<y$ to deduce the following.

\begin{lem}
If $\alpha,\beta,\theta_2$ satisfy the hypotheses of the last lemma and $\alpha<\delta<\beta$ then 
$$\sum_{u\leq p<y}w_pS(\ca_p,z)\leq XV(z)\left(\frac{e^{-\gamma}}{k\alpha}\int_\delta^\beta\left(\frac{1}{s}-\frac{1}{\beta}\right)F_2\left(\frac{\theta_2-s}{\alpha}\right)\,ds+o(1)\right).$$
\end{lem}

\begin{proof}
From the last lemma we obtain
$$\sum_{u\leq p<y}w_pS(\ca_p,z)\leq xV_2(z)\sum_{u\leq p<y}w_p\left(\frac{\nu_1(p)}{p}(F_2(s'_p)+O((\log x)^{-\frac{1}{6}}))\right).$$
We have 
$$\sum_{u\leq p<y}w_p\frac{\nu_1(p)}{p}(\log x)^{-\frac{1}{6}}\ll (\log x)^{-\frac{1}{6}}\sum_{u\leq p<y}\frac{1}{p}=o(1)$$
so it remains to evaluate 
$$\sum_{u\leq p<y}w_p\frac{\nu_1(p)}{p}F_2(s'_p)=\sum_{u\leq p<y}\frac{\nu_1(p)\log p}{p}h(p)$$
where 
$$h(t)=\left(\frac{1}{\log t}-\frac{1}{\log y}\right)F_2\left(\frac{\log(N^{\theta_2/t})}{\log z}\right).$$ 
Using partial summation and the estimate (\ref{nu1sum}) we obtain 
$$\sum_{u\leq p<y}w_p\frac{\nu_1(p)}{p}F_2(s'_p)=(\log y-\log u+O(1))h(y)-\int_u^y(\log t-\log u+O(1))h'(t)\,dt.$$
The contribution of the $O(1)$ errors can be dealt with by a very similar argument to the corresponding part of the proof of Lemma \ref{lem1dim}.  We therefore have 
\begin{eqnarray*}
\sum_{u\leq p<y}w_p\frac{\nu_1(p)}{p}F_2(s'_p)&=&(\log y-\log u)h(y)-\int_u^y(\log t-\log u)h'(t)\,dt+o(1)\\
&=&\int_u^y\frac{h(t)}{t}\,dt+o(1)\\
&=&\int_u^y\frac{1}{t}\left(\frac{1}{\log t}-\frac{1}{\log y}\right)F_2\left(\frac{\log(N^{\theta_2/t})}{\log z}\right)\,dt+o(1)\\  
&=&\int_\delta^\beta\left(\frac{1}{s}-\frac{1}{\beta}\right)F_2\left(\frac{\theta_2-s}{\alpha}\right)\,ds+o(1).\\  
\end{eqnarray*}
We conclude that 
$$\sum_{u\leq p<y}w_pS(\ca_p,z)\leq xV_2(z)\left(\int_\delta^\beta\left(\frac{1}{s}-\frac{1}{\beta}\right)F_2\left(\frac{\theta_2-s}{\alpha}\right)\,ds+o(1)\right).$$
To complete the proof we use  Lemma \ref{sumprods} to write 
$$V_2(z)=\frac{V(z)}{\log z}(e^{-\gamma}+o(1))$$
and we use the Prime Number Theorem to obtain 
$$x\sim X\log x\sim \frac{1}{k\alpha}X\log z.$$
\end{proof}

Combining all the estimates of this section we conclude that if $\alpha,\beta,\delta,\theta_1,\theta_2$ satisfy 
$$0<\alpha<\delta<\beta<\theta_2<\frac{1}{k}\text{ and }\delta<\theta_1<\frac{1}{2k}$$
then 
\begin{eqnarray*}
S(\ca,z)&\geq &XV(z)\left(f_1\left(\frac{\theta_1}{\alpha}\right)-\frac{1}{\eta}\left(\int_\alpha^\delta\left(\frac{1}{s}-\frac{1}{\beta}\right)F_1\left(\frac{\theta_1-s}{\alpha}\right)\,ds\right.\right.\\
&&\hspace{3cm}\left.\left.+\frac{e^{-\gamma}}{k\alpha}\int_\delta^\beta\left(\frac{1}{s}-\frac{1}{\beta}\right)F_2\left(\frac{\theta_2-s}{\alpha}\right)\,ds\right)+o(1)\right).\\
\end{eqnarray*}
It is clear that we should take $\theta_1,\theta_2$ as large as possible.  We therefore use continuity to see that for any $\epsilon>0$ we can choose $\theta_1$ sufficiently close to $\frac{1}{2k}$ and $\theta_2$ sufficiently close to $\frac{1}{k}$ to obtain 
\begin{eqnarray*}
S(\ca,z)&\geq &XV(z)\left(f_1\left(\frac{1}{2k\alpha}\right)-\frac{1}{\eta}\left(\int_\alpha^\delta\left(\frac{1}{s}-\frac{1}{\beta}\right)F_1\left(\frac{1-2ks}{2k\alpha}\right)\,ds\right.\right.\\
&&\hspace{3cm}\left.\left.+\frac{e^{-\gamma}}{k\alpha}\int_\delta^\beta\left(\frac{1}{s}-\frac{1}{\beta}\right)F_2\left(\frac{1-ks}{k\alpha}\right)\,ds\right)-\epsilon+o(1)\right).\\
\end{eqnarray*}
By (\ref{nu1prod}) and the Prime Number Theorem we know that 
$$XV(z)\gg \frac{x}{(\log x)^2}.$$
We may therefore deduce Theorem \ref{mainthm} for a given $k,r$ provided that we can find $\alpha,\delta,\beta$ satisfying 
$$0<\alpha<\delta<\beta<\frac{1}{k},\delta<\frac{1}{2k}\text{ and }\beta>\frac{1}{r+1}$$
for which 
$$f_1\left(\frac{1}{2k\alpha}\right)-\frac{1}{\eta}\left(\int_\alpha^\delta\left(\frac{1}{s}-\frac{1}{\beta}\right)F_1\left(\frac{1-2ks}{2k\alpha}\right)\,ds+\frac{e^{-\gamma}}{k\alpha}\int_\delta^\beta\left(\frac{1}{s}-\frac{1}{\beta}\right)F_2\left(\frac{1-ks}{k\alpha}\right)\,ds\right)>0.$$
Recalling that 
$$\eta=r+1-\frac{1}{\beta},$$
this is equivalent to 
$$r>\frac{1}{\beta}-1+\frac{1}{f_1(\frac{1}{2k\alpha})}\left(\int_\alpha^\delta\left(\frac{1}{s}-\frac{1}{\beta}\right)F_1\left(\frac{1-2ks}{2k\alpha}\right)\,ds+\frac{e^{-\gamma}}{k\alpha}\int_\delta^\beta\left(\frac{1}{s}-\frac{1}{\beta}\right)F_2\left(\frac{1-ks}{k\alpha}\right)\,ds\right).$$
We now let $\alpha_0=k\alpha$, $\beta_0=k\beta$, $\delta_0=k\delta$ and change variables in the integrals to write this as 
\begin{equation}\label{rcond}
r>\frac{k}{\beta_0}-1+\frac{1}{f_1(\frac{1}{2\alpha_0})}\left(\int_{\alpha_0}^{\delta_0}(\frac{1}{s}-\frac{1}{\beta_0})F_1\left(\frac{1-2s}{2\alpha_0}\right)\,ds+\frac{e^{-\gamma}}{\alpha_0}\int_{\delta_0}^{\beta_0}(\frac{1}{s}-\frac{1}{\beta_0})F_2\left(\frac{1-s}{\alpha_0}\right)\,ds\right).
\end{equation}
To find the permissible $r$ for a given $k$ it therefore only remains to calculate the minimum of this expression over the values of $\alpha_0,\beta_0,\delta_0$ which satisfy 
$$0<\alpha_0<\delta_0<\beta_0<1,\delta_0<\frac{1}{2}\text{ and }\frac{\beta_0}{k}>\frac{1}{r+1}.$$

\section{Proof of Theorem \ref{mainthm}}

To simplify the calculations we restrict $\alpha_0$ and $\delta_0$ so that all the sieve functions in (\ref{rcond}) may be given explicitly.  We have 
$$f_1(s)=A_1s^{-1}\log(s-1)\text{ for }2\leq s\leq 4,$$
$$F_1(s)=A_1s^{-1}\text{ for }s\leq 3$$
and 
$$F_2(s)=A_2s^{-2}\text{ for }s\leq \beta_2+1,$$
where 
$$A_1=2e^\gamma, \quad \beta_2=4.8333\ldots \quad \text{and}\quad A_2=43.496\ldots .$$
The values of $\beta_2$ and $A_2$ were computed using the formulae given by Friedlander and Iwaniec in \cite[Chapter 11]{opera}.  Note that the value of $A_2$ given by the table in \cite[Section  11.19]{opera} is incorrect by a factor of $2$ but our $\beta_2$ agrees with theirs.

We assume that 
$$\frac{1}{8}\leq \alpha_0\leq \frac{1}{4}$$
which implies that 
$$\frac{1-2\alpha_0}{2\alpha_0}\leq 3.$$
We also suppose that 
$$\frac{1-\delta_0}{\alpha_0}\leq \beta_2+1.$$
The condition (\ref{rcond}) then simplifies to 
$$r>\frac{k}{\beta_0}-1+\frac{1}{2A_1\log(\frac{1}{2\alpha_0}-1)}\left(2A_1\int_{\alpha_0}^{\delta_0}(\frac{1}{s}-\frac{1}{\beta_0})\frac{ds}{1-2s}+e^{-\gamma}A_2\int_{\delta_0}^{\beta_0}(\frac{1}{s}-\frac{1}{\beta_0})\frac{ds}{(1-s)^2}\right).$$
It is now clear that the optimal choice of $\delta_0$ is the root of 
$$\frac{2A_1}{1-2\delta_0}=\frac{e^{-\gamma}A_2}{(1-\delta_0)^2}$$
so it is $\delta_0=0.456\ldots$.  This choice satisfies $\delta_0<\frac{1}{2}$ as required.  We take $\alpha_0=\frac{1}{8}$, verifying that 
$$\frac{1-\delta_0}{\alpha_0}=4.344\ldots <\beta_2+1,$$ 
in which case (\ref{rcond}) becomes$r>r(k,\beta_0)$, where
$$r(k,\beta_0)=\frac{k}{\beta_0}-1+\frac{1}{2A_1\log 3}\left(2A_1\int_{\frac{1}{8}}^{\delta_0}\left(\frac{1}{s}-\frac{1}{\beta_0}\right)\frac{ds}{1-2s}+e^{-\gamma}A_2\int_{\delta_0}^{\beta_0}\left(\frac{1}{s}-\frac{1}{\beta_0}\right)\frac{ds}{(1-s)^2}\right).$$
To establish Theorem \ref{mainthm} we must minimise this over $\beta_0\in (\delta_0,1)$, verifying that the resulting $\beta_0,r$ satisfy 
\begin{equation}\label{betacond}
\frac{\beta_0}{k}>\frac{1}{r+1}.
\end{equation}
We note that both integrals in the above constraint can be evaluated explicitly.  In addition it can be shown by calculus that the optimal $\beta_0$ is of the form 
$$1-\frac{1}{c_1k+c_2}$$
for certain constants $c_1,c_2$.  We do not include all the details as it is not necessary to prove that our choices for $\beta_0$ are optimal.  For small $k$ we used a computer to symbolically evaluate the integrals and find the optimal $\beta_0$.  The results are given in Table \ref{betachoices}, they complete the proof of Theorem \ref{mainthm} for such $k$ (the condition (\ref{betacond}) can be verified in each instance).

\begin{table}
\centering

\label{betachoices}
\begin{tabular}{r|rrrrrrrrr}
  $k$ & 2 & 3 & 4 & 5 & 6 & 7 & 8 & 9 & 10\\
\hline
$\beta_0$ & 0.54 & 0.60 & 0.64 & 0.68 & 0.71 & 0.73 & 0.76 & 0.77 & 0.79\\
$r(k,\beta_0)$ & 4.2 & 5.9 & 7.6 & 9.1 & 10.5 & 11.9 & 13.2 & 14.5 & 15.8\\
$r$& 5 & 6 & 8 & 10 & 11 & 12 & 14 & 15 & 16\\ 
\end{tabular}
\caption{Results for small $k$}
\end{table}

It remains to deal with large $k$.  Observe that if we chose $\beta_0<1$ independent of $k$ we would have 
$$r(k,\beta_0)=\frac{k}{\beta_0}+c_{\beta_0},$$
for a constant $c_{\beta_0}$ depending on $\beta_0$. Instead we choose 
$$\beta_0=1-\frac{1}{k}$$
which, although not being completely optimal, is sufficient for our purposes.  We then have 
$$\frac{k}{\beta_0}=\frac{k^2}{k-1}=k+O(1),$$ 
$$\int_{\frac{1}{8}}^{\delta_0}\left(\frac{1}{s}-\frac{1}{\beta_0}\right)\frac{ds}{1-2s}=O(1)$$
and 
\begin{eqnarray*}
\int_{\delta_0}^{\beta_0}\left(\frac{1}{s}-\frac{1}{\beta_0}\right)\frac{ds}{(1-s)^2}&=&\int_{\delta_0}^{\beta_0}\left(\frac{1}{s}+\frac{1}{1-s}+(1-1/\beta_0)\frac{1}{(1-s)^2}\right)\,ds\\
&=&[-\log(1-s)+\frac{1-1/\beta_0}{1-s}]_{\delta_0}^{\beta_0}+O(1)\\
&=&-\log(1-\beta)+O(1)\\
&=&\log k+O(1).\\
\end{eqnarray*}
We conclude that 
$$r(k,1-\frac{1}{k})=k+\frac{e^{-\gamma}A_2}{2A_1\log 3}\log k+O(1)=k+c\log k+O(1),$$
with 
$$c=\frac{e^{-\gamma}A_2}{2A_1\log 3}=\frac{e^{-2\gamma}A_2}{4\log3}=3.120\ldots.$$
It only remains to verify the condition (\ref{betacond}), for which we require 
$$\frac{1}{k}-\frac{1}{k^2}>\frac{1}{k+c\log k+O(1)}.$$
This holds for sufficiently large $k$ since 
$$\frac{1}{k}-\frac{1}{k+c\log k}\geq \frac{c\log k}{k^2}.$$

\section{Possible Improvements}

There are several small improvements to our method which we chose not to implement.  Most significantly, Diamond and Halberstam \cite{diamondhalberstam} describe a $2$-dimensional sieve which is better than the beta-sieve.  Their upper bound function satisfies 
$$F_2(s)=\frac{A_2}{s^2}\text{ for }s\leq 2$$
with $A_2=25.377\ldots$.  For $s>2$ the expressions for $F_2(s)$ are not so simple and therefore the evaluation of (\ref{rcond}) would require numerical integration and optimisation.  This would certainly lead to an improvement to the value of $c$ in Theorem \ref{mainthm} and possibly also to better results for some small degrees.

Recent work of Zhang \cite{zhang} and Polymath \cite{Polymath8a} has given an improved level of distribution for the primes in arithmetic progressions to smooth moduli.  This could be used to slightly improve our lower bound for $S(\ca,z)$ by means of the Buchstab identity 
$$S(\ca,z)=S(\ca,w)-\sum_{w\leq p<z}S(\ca_p,p).$$
If $w$ is chosen to be a suitably small power of $x$ then the results of Zhang and Polymath would apply to the remainder term when estimating $S(\ca,w)$, thereby enabling us to sieve beyond $\sqrt{x}$ and get a better bound.  However, since the sieve function $f_1(s)$ converges rapidly to $1$ as $s\rightarrow \infty$ the improvement would be extremely small.

Finally it is worth commenting on the reason for the $\log\deg f$ term in Theorem \ref{mainthm}, especially since no such term appears in Richert's result $r\geq \deg f+1$ for the values of $f(n)$.  The difference is caused by the fact that, as $s\rightarrow 0$,  $F_2(s)$ grows like $s^{-2}$ whereas $F_1(s)$ only grows like $s^{-1}$.  Our weights $w_p$ do not decrease sufficiently rapidly to handle this behaviour of $F_2$.  One might therefore hope that our result could be improved by a better choice of weight.  Recall that as well as making the sum $S$ as large as possible the weights must be chosen in such a way that a result like Lemma \ref{lemweighted} can be established.

\newpage

\addcontentsline{toc}{section}{References} 
\bibliographystyle{plain}
\bibliography{../biblio}

\begin{thebibliography}{1}

\bibitem{chen}
J.-R. Chen.
\newblock On the representation of a larger even integer as the sum of a prime
  and the product of at most two primes.
\newblock {\em Sci. Sinica}, 16:157--176, 1973.

\bibitem{diamondhalberstam}
H.~G. Diamond and H.~Halberstam.
\newblock {\em A higher-dimensional sieve method}, volume 177 of {\em Cambridge
  Tracts in Mathematics}.
\newblock Cambridge University Press, Cambridge, 2008.
\newblock With an appendix (``Procedures for computing sieve functions'') by
  William F. Galway.

\bibitem{elliott}
P.~D. T.~A. Elliott and H.~Halberstam.
\newblock A conjecture in prime number theory.
\newblock In {\em Symposia {M}athematica, {V}ol. {IV} ({INDAM}, {R}ome,
  1968/69)}, pages 59--72. Academic Press, London, 1970.

\bibitem{opera}
J.~Friedlander and H.~Iwaniec.
\newblock {\em Opera de cribro}, volume~57 of {\em American Mathematical
  Society Colloquium Publications}.
\newblock American Mathematical Society, Providence, RI, 2010.

\bibitem{iwaniecn21}
H.~Iwaniec.
\newblock Almost-primes represented by quadratic polynomials.
\newblock {\em Invent. Math.}, 47(2):171--188, 1978.

\bibitem{Polymath8a}
D.~H.~J. Polymath.
\newblock New equidistribution estimates of {Zhang} type, and bounded gaps
  between primes.
\newblock arXiv:1402.0811.

\bibitem{richert}
H.-E. Richert.
\newblock Selberg's sieve with weights.
\newblock {\em Mathematika}, 16:1--22, 1969.

\bibitem{zhang}
Y.~Zhang.
\newblock Bounded gaps between primes.
\newblock {\em Ann. of Math. (2)}, 179(3):1121--1174, 2014.

\end{thebibliography}

\bigskip
\bigskip

Centre de recherches math\'ematiques,

Universit\'e de Montr\'eal,

Pavillon Andr\'e-Aisenstadt,

2920 Chemin de la tour, Room 5357,

Montr\'eal (Qu\'ebec) H3T 1J4

\bigskip 

{\tt alastair.j.irving@gmail.com}

\end{document}